\documentclass[preprint,12pt]{elsarticle}
\usepackage[T1]{fontenc}
\usepackage{lmodern,amsmath,amssymb,amsthm,microtype}
\usepackage{booktabs}
\usepackage{needspace}
\usepackage[skip=5pt plus 1pt minus 1pt,indent=0pt]{parskip}
\AtBeginEnvironment{frontmatter}{\setlength{\parskip}{0pt}}
\usepackage{xcolor}
\usepackage[colorlinks=true,allcolors=blue!45!black]{hyperref}
\hypersetup{pdftitle={Comment on ``The polarization hierarchy for polynomial optimization over convex bodies, with applications to nonnegative matrix rank''},pdfauthor={Julius A. Zeiss}}
\journal{Linear Algebra and its Applications}
\makeatletter\def\NAT@cmt{,~}\makeatother
\makeatletter
\patchcmd{\pprintMaketitle}{\vskip36pt}{\vskip10pt}{}{}
\makeatother

\numberwithin{equation}{section}
\newcommand{\blockparagraphs}{%
  \setlength{\parskip}{0pt}%
  \setlength{\parindent}{1em}%
  \setlength{\listparindent}{1em}%
}
\newtheoremstyle{statement}
  {7pt plus 1pt minus 1pt}{7pt plus 1pt minus 1pt}
  {\itshape\blockparagraphs}{0pt}{\bfseries}{.}{.5em}{}
\newtheoremstyle{explanation}
  {7pt plus 1pt minus 1pt}{7pt plus 1pt minus 1pt}
  {\normalfont\blockparagraphs}{0pt}{\bfseries}{.}{.5em}{}
\theoremstyle{statement}
\newtheorem{theorem}{Theorem}[section]
\newtheorem{proposition}[theorem]{Proposition}
\newtheorem{corollary}[theorem]{Corollary}
\theoremstyle{explanation}
\newtheorem{example}[theorem]{Example}
\newtheorem{remark}[theorem]{Remark}
\AtBeginEnvironment{theorem}{\Needspace{5\baselineskip}}
\AtBeginEnvironment{proposition}{\Needspace{5\baselineskip}}
\AtBeginEnvironment{corollary}{\Needspace{5\baselineskip}}
\AtBeginEnvironment{example}{\Needspace{5\baselineskip}}
\AtBeginEnvironment{remark}{\Needspace{5\baselineskip}}
\AtBeginEnvironment{proof}{\blockparagraphs}
\AtBeginEnvironment{example}{\pushQED{\qed}}
\AtEndEnvironment{example}{\popQED}
\AtBeginEnvironment{remark}{\pushQED{\qed}}
\AtEndEnvironment{remark}{\popQED}

\newcommand{\R}{\mathbb R}
\newcommand{\id}{\mathrm{id}}
\newcommand{\tr}{\operatorname{tr}}
\newcommand{\rank}{\operatorname{rank}}
\newcommand{\conv}{\operatorname{conv}}
\newcommand{\supp}{\operatorname{supp}}
\newcommand{\ran}{\operatorname{ran}}
\newcommand{\Sym}{\operatorname{Sym}}
\newcommand{\Ext}{\operatorname{Ext}}
\newcommand{\PSD}{\mathbb S_+}
\newcommand{\tmax}{\mathbin{\widehat\otimes}}
\newcommand{\Fcal}{\mathcal F}
\newcommand{\Kn}[1]{\mathcal K_{#1}}
\newcommand{\Sreg}{\mathcal S}
\newcommand{\Sn}{\mathfrak S}
\newcommand{\norm}[1]{\lVert#1\rVert}
\newcommand{\Hc}[1]{$(\mathrm H_{#1})$}
\newcommand{\one}{\mathbf 1}

\begin{document}

\begin{frontmatter}
\title{Comment on ``The polarization hierarchy for polynomial optimization over convex bodies, with applications to nonnegative matrix rank''}

\author{Julius A. Zeiss}
\ead{jzeiss@physik.rwth-aachen.de}
\address{Institute for Quantum Information, RWTH Aachen University, Aachen, Germany}

\begin{abstract}
We revisit the feasibility and convergence claims for the polarization hierarchy of Pl\'avala, Ligthart and Gross [Linear Algebra Appl.\ 723 (2025) 15--32]. Counterexamples show that the stated assumption is insufficient for the feasibility characterization and for convergence of the hierarchy with the polarization constraint. We establish corrected results under a stronger algebraic condition and supply a missing compactness argument.
\end{abstract}

\end{frontmatter}

\section{Introduction}\label{sec:intro}
Let $K_A$ and $K_B$ be compact convex sets, and let $f:K_A\times K_B\to V$ and $p:K_A\times K_B\to\R$ be affine in each argument. The polarization hierarchy of~\cite{PLG} is a sequence of convex relaxations of the problem
\begin{equation}\label{eq:original}
\Fcal=\{(a,b)\in K_A\times K_B:f(a,b)=0\},\qquad p^*=\min_{(a,b)\in\Fcal}p(a,b),
\end{equation}
formulated over maximal tensor products of copies of $K_A$ and $K_B$. If $K_A$ and $K_B$ are polytopes, each relaxation is a linear program. The constraint $f(a,b)=0$ enters through a linear map $\Pi:V\otimes V\to W$, the \emph{polarization map}, which is assumed to satisfy
\begin{equation*}\label{eq:H2}\tag{$\mathrm H_2$}
\Pi(v\otimes v+w\otimes w)=0\quad\Longrightarrow\quad v=w=0\qquad(v,w\in V).
\end{equation*}
Under this condition, \cite[Thm.~3]{PLG} characterizes $\Fcal\ne\varnothing$ by bicompatible sequences, and \cite[Thm.~5]{PLG} asserts that the relaxations converge to $p^*$.

The de Finetti theorem \cite[Lem.~2]{PLG} represents a bicompatible sequence by a probability measure $\mu$ on $K_A\times K_B$. For a sequence satisfying the polarization constraint, this representation gives
\[
\Pi(M)=0,\qquad M=\int f(a,b)\otimes f(a,b)\,d\mu(a,b).
\]
Here $f(a,b)$ is the constraint residual, which vanishes exactly at feasible pairs. Fixing an inner product on $V$, we regard $M$ as its second-moment operator under $\mu$. Each outer product $f(a,b)\otimes f(a,b)$ is positive semidefinite, since
\[
\bigl\langle z,\,[f(a,b)\otimes f(a,b)]z\bigr\rangle=\langle z,f(a,b)\rangle^2\ge0\qquad(z\in V).
\]
Its average $M$ is therefore positive semidefinite. Its trace is the average squared constraint violation:
\[
\tr M=\int\norm{f(a,b)}^2\,d\mu(a,b).
\]
If $M=0$, then $f(a,b)=0$ for $\mu$-almost every pair $(a,b)$, so the original constraint has a solution. The proof of \cite[Thm.~3]{PLG} therefore seeks to deduce $M=0$ from $\Pi(M)=0$.

The difficulty lies in this last implication. Since $M$ is positive semidefinite, its spectral decomposition gives
\[
M=\sum_{j=1}^{r}v_j\otimes v_j,\qquad r=\rank M\le\dim V.
\]
By linearity of $\Pi$, the condition $\Pi(M)=0$ becomes
\[
\sum_{j=1}^{r}\Pi(v_j\otimes v_j)=0.
\]
Condition \Hc{2} forces every $v_j$ to vanish when $r\le2$, but does not exclude cancellation when $r\ge3$. Thus \Hc{2} does not, in general, justify the implication $\Pi(M)=0\Rightarrow M=0$. We show that this gap is genuine and close it:
\begin{itemize}
\item Exs.~\ref{ex:feasibility} and~\ref{ex:optimization} are counterexamples to \cite[Thm.~3]{PLG} and to the first assertion of \cite[Thm.~5]{PLG} for a polarization map satisfying \Hc{2}.
\item Both statements hold if the kernel of $\Pi$ contains no nonzero positive semidefinite operator (Prop.~\ref{prop:equiv}, Thms.~\ref{thm:polar} and~\ref{thm:convergence}). Equivalently, $\Pi(\sum_{j=1}^{D}v_j\otimes v_j)=0$ must imply $v_1=\cdots=v_D=0$, where $D=\dim V$. The number $D$ cannot be lowered in general (Prop.~\ref{prop:sharp}), but a smaller number suffices if $\Pi$ has small rank (Prop.~\ref{prop:rank}). This strengthened condition is necessary for \cite[Thm.~3]{PLG} to hold for all $K_A$, $K_B$ and $f$ (Prop.~\ref{prop:necessary}).
\item The proof of \cite[Thm.~5]{PLG} takes limits of two-copy marginals only and does not justify that the limit extends to a bicompatible sequence (Ex.~\ref{ex:parity}). Thm.~\ref{thm:convergence} supplies the missing argument.
\item The three polarization maps proposed in~\cite{PLG} already satisfy this strengthened condition (Rem.~\ref{rem:maps}), so convergence holds for these choices. The stronger hierarchy of \cite[Thm.~4]{PLG}, which does not use a polarization map, also converges (Cor.~\ref{cor:strong}). The numerical results in \cite[Fig.~1]{PLG} are unaffected (Rem.~\ref{rem:application}).
\end{itemize}
Sec.~\ref{sec:lp} verifies Ex.~\ref{ex:optimization} by linear programming.

\section{Preliminaries}\label{sec:prelim}
All vector spaces are real and finite-dimensional. A \emph{state space} is a nonempty compact convex subset $K$ of a vector space. We write $A(K)$ for the space of affine functions on $K$ and $\one_K$ for the constant function $1$, called the \emph{unit effect}. We identify $x\in K$ with the evaluation functional in $A(K)^*$. For state spaces $K_A,K_B$ let $E=A(K_A)^*\otimes A(K_B)^*$, and let $K_A\tmax K_B\subset E$ be the maximal tensor product \cite[Eq.~(4)]{PLG}.

As in~\cite{PLG}, a map on $K_A\times K_B$ that is affine in each argument is identified with the linear map on $E$ that agrees with it on the product states $a\otimes b$. This applies to $f$ and $p$. We assume $\Fcal\ne\varnothing$ whenever $p^*$ is used; then the minimum in~\eqref{eq:original} is attained, since $\Fcal$ is compact and $p$ is continuous.

We fix an arbitrary inner product on the constraint space $V$ and identify $u\otimes v\in V\otimes V$ with the operator
\[
x\longmapsto\langle v,x\rangle u.
\]
Under this identification, the space $\Sym^2(V)$ of symmetric tensors corresponds to the space of self-adjoint operators on $V$, and $\PSD(V)$ denotes the cone of positive semidefinite operators. In particular, $v\otimes v$ is positive semidefinite for every $v\in V$.

We write $\tr$ for the ordinary trace of operators on $V$. Under the above identification,
\[
\tr(u\otimes v)=\langle v,u\rangle,\qquad \tr(v\otimes v)=\norm v^2.
\]

For $n\ge1$ let
\[
\Kn n=K_A^{\tmax n}\tmax K_B^{\tmax n},
\]
with the tensor factors ordered as $A_1\cdots A_nB_1\cdots B_n$, and let $\Sreg_n$ be the linear map that reorders them as $A_1B_1\cdots A_nB_n$ \cite[Eq.~(8)]{PLG}. For $k\le n$, the marginalization map $R_{n,k}$, also called the partial trace, applies the unit effect $\one_{K_A}$ to the discarded factors $A_{k+1},\dots,A_n$ and $\one_{K_B}$ to $B_{k+1},\dots,B_n$; it maps $\Kn n$ into $\Kn k$. For $\sigma,\tau$ in the symmetric group $\Sn_n$, the maps $S_\sigma^A$ and $S_\tau^B$ permute the factors $A_1,\dots,A_n$ and $B_1,\dots,B_n$, respectively. A sequence $(y_n)_{n\ge1}$ with $y_n\in\Kn n$ is \emph{bicompatible} \cite[Def.~1]{PLG} if
\begin{equation}\label{eq:bicompatible}
(S_\sigma^A\otimes S_\tau^B)y_n=y_n\quad(\sigma,\tau\in\Sn_n)\qquad\text{and}\qquad R_{n,k}y_n=y_k\quad(k\le n).
\end{equation}
By \cite[Lem.~2]{PLG}, every bicompatible sequence has a Borel probability measure $\mu$ on $K_A\times K_B$ such that
\begin{equation}\label{eq:definetti}
y_n=\int a^{\otimes n}\otimes b^{\otimes n}\,d\mu(a,b)\qquad(n\ge1).
\end{equation}
We call $\mu$ a \emph{representing measure} of $(y_n)$.

\paragraph{The hierarchies}
Let $\Pi:V\otimes V\to W$ be linear. For $n\ge2$ let $\Omega_n^\Pi$ be the set of $y\in\Kn n$ such that
\begin{equation}\label{eq:hierarchy}
(S_\sigma^A\otimes S_\tau^B)y=y\quad(\sigma,\tau\in\Sn_n)\qquad\text{and}\qquad\Pi\bigl((f\otimes f)\Sreg_2(R_{n,2}y)\bigr)=0,
\end{equation}
and let
\[
p_n^\Pi=\inf\{p(R_{n,1}y):y\in\Omega_n^\Pi\},
\]
with $\inf\varnothing=+\infty$. This is the relaxation \cite[Eq.~(28)]{PLG}, with the order of the tensor factors and the partial traces written explicitly. The stronger relaxation \cite[Eq.~(29)]{PLG} replaces the second constraint in~\eqref{eq:hierarchy} by
\begin{equation}\label{eq:strong}
(f\otimes\id_E^{\otimes(n-1)})\Sreg_n(y)=0\qquad\text{in }V\otimes E^{\otimes(n-1)}.
\end{equation}
We denote its feasible set by $\Omega_n$ and its value by $q_n$.

If $(a,b)\in\Fcal$, then $y_n=a^{\otimes n}\otimes b^{\otimes n}$ satisfies~\eqref{eq:hierarchy} and~\eqref{eq:strong}, because
\[
\begin{aligned}
(f\otimes f)\Sreg_2(y_2)&=f(a,b)\otimes f(a,b),\\
(f\otimes\id_E^{\otimes(n-1)})\Sreg_n(y_n)&=f(a,b)\otimes(a\otimes b)^{\otimes(n-1)}.
\end{aligned}
\]
Moreover, $p(R_{n,1}y_n)=p(a,b)$. Hence, for every linear map $\Pi$,
\begin{equation}\label{eq:upper}
\Fcal\ne\varnothing\quad\Longrightarrow\quad p_n^\Pi\le p^*\ \text{ and }\ q_n\le p^*\qquad(n\ge2).
\end{equation}
In particular, $\Omega_n^\Pi=\varnothing$ for some $n$ certifies $\Fcal=\varnothing$, whatever $\Pi$ is.

Thm.~3 of~\cite{PLG} states that if $\Pi$ satisfies \Hc{2}, then $\Fcal\ne\varnothing$ if and only if there is a bicompatible sequence $(y_n)$ with $\Pi((f\otimes f)\Sreg_2(y_2))=0$. Thm.~5 of~\cite{PLG} states that, for such $\Pi$, $\lim_{n\to\infty}p_n^\Pi=p^*$ and $\lim_{n\to\infty}q_n=p^*$.

\section{Condition \texorpdfstring{$(\mathrm H_2)$}{(H2)} is insufficient}\label{sec:counter}
Let $(y_n)$ be bicompatible with $\Pi((f\otimes f)\Sreg_2(y_2))=0$, and let $\mu$ be a representing measure. Inserting~\eqref{eq:definetti} gives
\begin{equation}\label{eq:moment}
\Pi(M)=0,\qquad M=\int f(a,b)\otimes f(a,b)\,d\mu(a,b).
\end{equation}
The proof of \cite[Thm.~3]{PLG} then considers the set $W^+=\{\Pi(v\otimes v):v\in V\}$. Condition \Hc{2} implies $W^+\cap(-W^+)=\{0\}$, and \cite[p.~20, after Eq.~(17)]{PLG} concludes, citing \cite[Prop.~B.8]{Plavala}, that the dual cone $(W^+)^*=\{\varphi\in W^*:\varphi\ge0\text{ on }W^+\}$ is generating, that is, spans $W^*$.

However, $(W^+)^*$ is also the dual cone of the closed convex cone generated by $W^+$. Indeed, a linear functional nonnegative on $W^+$ is also nonnegative on all nonnegative linear combinations of its elements and on their limits. These combinations can involve more than two generators, whose cancellation is not excluded by \Hc{2}. Thus \Hc{2} does not imply that the closed convex cone generated by $W^+$ is pointed. Since a closed convex cone is pointed if and only if its dual cone is generating, this pointedness is exactly what the argument requires; given \Hc{2}, it is equivalent to the corrected condition of Sec.~\ref{sec:condition} (Prop.~\ref{prop:equiv}(v)).

Ex.~\ref{ex:feasibility} exhibits this obstruction with a map satisfying \Hc{2} for which the convex cone generated by $W^+$ is all of $W$, so that $(W^+)^*=\{0\}$. Using this map, we construct an infeasible problem that nevertheless admits a bicompatible sequence satisfying the polarization constraint, contradicting \cite[Thm.~3]{PLG}. Ex.~\ref{ex:optimization} adapts this construction to a feasible optimization problem, disproving the first convergence assertion in \cite[Thm.~5]{PLG}.

For this construction and further results throughout the paper, we use the following family of projections onto the symmetric traceless matrices:
\begin{equation}\label{eq:PiD}
\Pi_D(X)=\frac{X+X^{\mathsf T}}{2}-\frac{\tr X}{D}\,I_D\qquad(D\ge2).
\end{equation}
The subscript $D$ denotes the matrix dimension. An unsubscripted $\Pi$ denotes a general polarization map.

\begin{example}[Counterexample to {\cite[Thm.~3]{PLG}}]\label{ex:feasibility}
Let $V=\R^3$, identify $V\otimes V$ with $\R^{3\times3}$ via $v\otimes w\mapsto vw^{\mathsf T}$, let $W$ be the space of symmetric traceless $3\times3$ matrices, and take $\Pi=\Pi_3$ from~\eqref{eq:PiD}.
If $\Pi_3(vv^{\mathsf T}+ww^{\mathsf T})=0$, then $vv^{\mathsf T}+ww^{\mathsf T}=cI_3$ with $c\ge0$; since the left side has rank at most $2$, $c=0$ and $v=w=0$. Thus $\Pi_3$ satisfies \Hc{2}. On the other hand,
\[
\sum_{i=1}^3\Pi_3(e_ie_i^{\mathsf T})=\Pi_3(I_3)=0,\qquad\Pi_3(e_ie_i^{\mathsf T})\ne0.
\]
In fact, the convex cone generated by $W^+$ is $W$: for $S\in W$ choose $t\ge0$ with $S+tI_3$ positive semidefinite and write $S+tI_3=\sum_i\lambda_iu_iu_i^{\mathsf T}$ with $\lambda_i\ge0$; then
\[
S=\Pi_3(S+tI_3)=\sum_i\Pi_3(\lambda_iu_iu_i^{\mathsf T}).
\]

Now let $K_A=\conv\{e_1,e_2,e_3\}\subset\R^3$, let $K_B=\{b_0\}$ be a singleton, and let $f(a,b)=a$. Every $a\in K_A$ has nonnegative coordinates summing to $1$, so $f(a,b)=a$ cannot vanish and $\Fcal=\varnothing$. Here $A(K_A)^*\cong\R^3$, with $\one_{K_A}(x)=\sum_ix_i$. Define the sequence $y_n=\frac13\sum_{i=1}^3e_i^{\otimes n}\otimes b_0^{\otimes n}$. This sequence chooses one of the three vertices uniformly and repeats the chosen vertex in every $A$-copy, with $b_0$ in every $B$-copy. Independent permutations leave each term unchanged, and marginalization gives the same construction at lower order. Thus $(y_n)$ is bicompatible, and
\[
(f\otimes f)\Sreg_2(y_2)=\frac13\sum_{i=1}^3e_ie_i^{\mathsf T}=\frac13I_3\in\ker\Pi_3.
\]
Hence \cite[Thm.~3]{PLG} fails.
\end{example}

We now turn the preceding construction into a feasible optimization problem for which the hierarchy fails to converge.

\begin{example}[Counterexample to {\cite[Thm.~5]{PLG}}]\label{ex:optimization}
We modify Ex.~\ref{ex:feasibility} by adding the origin to $K_A$ and introducing an objective, while retaining the polarization map $\Pi_3$. Let $K_B$ be a singleton, which we suppress, let $K_A=\conv\{0,e_1,e_2,e_3\}\subset\R^3$, and let $f(x)=x$ and $p(x)=-\sum_{i=1}^3x_i$. Then $\Fcal=\{0\}$ and $p^*=0$. We identify $A(K_A)^*$ with $\R\times\R^3$ and $x\in K_A$ with $(1,x)$; then $\one_{K_A}(t,x)=t$, $f(t,x)=x$, $p(t,x)=-\sum_ix_i$, and $K_A$ is the simplex with vertices $\sigma_0=(1,0)$ and $\sigma_i=(1,e_i)$, $i=1,2,3$. The additional coordinate records normalization; in particular, the newly added origin is represented by the normalized state $\sigma_0=(1,0)$.

Adding the origin creates a feasible point but leaves the three original vertices available. The same uniform mixture of their tensor powers continues to satisfy the polarization constraint, while each of these vertices has objective value $-1$. The sequence $y_n=\frac13\sum_{i=1}^3\sigma_i^{\otimes n}$ is bicompatible with
\[
(f\otimes f)\Sreg_2(y_2)=\frac13I_3\in\ker\Pi_3,
\]
so $y_n\in\Omega_n^{\Pi_3}$, and $p(R_{n,1}y_n)=-1$. Since $p\ge-1$ on $K_A$,
\begin{equation}\label{eq:badvalue}
p_n^{\Pi_3}=-1\quad(n\ge2),\qquad\lim_{n\to\infty}p_n^{\Pi_3}=-1<0=p^*.
\end{equation}
As~\eqref{eq:hierarchy} is \cite[Eq.~(28)]{PLG} and $\Pi_3$ satisfies \Hc{2}, this contradicts the first assertion of \cite[Thm.~5]{PLG}. For comparison, $p_n^{\id}=-1/n\to p^*$ (Sec.~\ref{sec:lp}). The example uses only the constraints in~\eqref{eq:hierarchy}; adding $f(R_{n,1}y)=0$ or~\eqref{eq:strong} defines a different relaxation.

Note that Ex.~\ref{ex:feasibility} also illustrates the insufficiency of \Hc{2} for feasibility problems: the polarization hierarchy remains feasible at every level, even though the original problem is infeasible.
\end{example}

\begin{remark}\label{rem:singleton}
The choice of a singleton $K_B$ simplifies Exs.~\ref{ex:feasibility} and~\ref{ex:optimization}, but both constructions extend to two nontrivial systems with coupled constraints.

In Ex.~\ref{ex:feasibility}, take $K_B=K_A$ and define
\[
f(a,b)=\sum_{i,j=1}^3a_ib_j\,e_{i+j},\qquad e_{k+3}=e_k.
\]
Since the entries of $f(a,b)$ sum to $1$, the original problem is infeasible. The sequence
\[
y_n=\frac13\sum_{i=1}^3e_i^{\otimes n}\otimes e_i^{\otimes n}
\]
is nevertheless bicompatible and satisfies the polarization constraint, because its second-moment operator is again $I_3/3\in\ker\Pi_3$.

\Needspace{6\baselineskip}
In Ex.~\ref{ex:optimization}, the corresponding modification is
\[
K_B=K_A,\qquad f(a,b)=(a_1b_1,a_2b_2,a_3b_3),\qquad p(a,b)=-\sum_{i=1}^3a_ib_i.
\]
Here $p^*=0$, while
\[
y_n=\frac13\sum_{i=1}^3\sigma_i^{\otimes n}\otimes\sigma_i^{\otimes n}
\]
remains feasible at every hierarchy level and gives $p_n^{\Pi_3}\le-1$.

Thus the failure occurs both for local affine constraints, as in the original examples, and for constraints coupling the two systems. The local case also connects to~\cite{ZKSP}, whose quantitative convergence results use a finite de Finetti theorem for constraints of the form $f(a,b)=f_A(a)\oplus f_B(b)$. Exs.~\ref{ex:feasibility} and~\ref{ex:optimization} show that this local structure alone does not make \Hc{2} sufficient.
\end{remark}

\section{The corrected condition}\label{sec:condition}
Let $D=\dim V\ge1$ (the case $V=\{0\}$ is trivial). For $k\ge1$ consider the condition
\begin{equation*}\label{eq:Hk}\tag{$\mathrm H_k$}
\Pi\Bigl(\sum_{j=1}^kv_j\otimes v_j\Bigr)=0\quad\Longrightarrow\quad v_1=\cdots=v_k=0\qquad(v_1,\dots,v_k\in V).
\end{equation*}
The corrected hypothesis is the case $k=D=\dim V$, denoted by \Hc{D}.
Since some $v_j$ may vanish, \Hc{k} implies \Hc{j} for $j\le k$.
The following proposition identifies the central algebraic condition underlying the corrected feasibility and convergence theorems and gives several equivalent formulations.

\begin{proposition}\label{prop:equiv}
For a linear map $\Pi:V\otimes V\to W$ the following are equivalent:
\begin{enumerate}
\item[\upshape(i)] $\Pi$ satisfies \Hc{D};
\item[\upshape(ii)] $\Pi$ satisfies \Hc{k} for every $k\ge1$;
\item[\upshape(iii)] $\ker\Pi\cap\PSD(V)=\{0\}$;
\item[\upshape(iv)] there are $\ell\in W^*$ and $c>0$ such that $\ell(\Pi(v\otimes v))\ge c\norm v^2$ for all $v\in V$;
\item[\upshape(v)] $\Pi$ satisfies \Hc{1}, and the closed convex cone generated by $W^+$ is pointed.
\end{enumerate}
\end{proposition}

\begin{proof}
\setlength{\abovedisplayskip}{6pt plus 1pt minus 1pt}
\setlength{\belowdisplayskip}{6pt plus 1pt minus 1pt}
\setlength{\abovedisplayshortskip}{6pt plus 1pt minus 1pt}
\setlength{\belowdisplayshortskip}{6pt plus 1pt minus 1pt}
(i)$\Rightarrow$(iii): Let $M\in\ker\Pi\cap\PSD(V)$. By the spectral theorem, write $M=\sum_{j=1}^{D}\lambda_j u_j\otimes u_j$, where $\lambda_j\ge0$, and set $v_j=\sqrt{\lambda_j}\,u_j$. Then
\[
\Pi\Bigl(\sum_{j=1}^{D}v_j\otimes v_j\Bigr)=\Pi(M)=0,
\]
so \Hc{D} implies $v_1=\cdots=v_D=0$, and hence $M=0$.

\smallskip\noindent
(iii)$\Rightarrow$(ii): Fix an arbitrary $k\ge1$ and vectors $v_1,\ldots,v_k\in V$ such that
\[
\Pi(M)=0,\qquad M=\sum_{j=1}^{k}v_j\otimes v_j.
\]
Each outer product $v_j\otimes v_j$ is positive semidefinite, so their sum $M$ is positive semidefinite as well. Since $M$ also belongs to $\ker\Pi$, condition (iii) gives $M=0$. To conclude that the individual vectors vanish, take the trace:
\[
0=\tr M=\sum_{j=1}^{k}\norm{v_j}^2.
\]
This is a sum of nonnegative numbers, so every $v_j=0$. Since $k$ was arbitrary, $\Pi$ satisfies \Hc{k} for every $k\ge1$.

\smallskip\noindent
(ii)$\Rightarrow$(i) is trivial.

\smallskip\noindent
(iii)$\Rightarrow$(iv): The positive semidefinite operators of trace one form a compact convex set. Since $\Pi$ is linear, its image
\[
Q=\Pi\bigl(\{M\in\PSD(V):\tr M=1\}\bigr)
\]
is also compact and convex. Moreover, $0\notin Q$ by (iii).

Fix an inner product on $W$. Compactness gives a point $q_0\in Q$ of minimal norm, and $q_0\ne0$. For every $q\in Q$, convexity gives $q_0+t(q-q_0)\in Q$ for $0\le t\le1$. Minimality of $q_0$ therefore implies
\[
0\le\left.\frac{d}{dt}\right|_{t=0^+}\norm{q_0+t(q-q_0)}^2
=2\langle q_0,q-q_0\rangle.
\]
Thus the linear functional $\ell(w)=\langle q_0,w\rangle$ satisfies
\[
\ell(q)\ge\norm{q_0}^2=:c>0\qquad(q\in Q).
\]
For $v\ne0$, the operator $(v\otimes v)/\norm v^2$ is positive semidefinite with trace one, so
\[
\ell(\Pi(v\otimes v))\ge c\norm v^2.
\]
For $v=0$ this is immediate, completing the proof.

\smallskip\noindent
(iv)$\Rightarrow$(ii): Fix $k\ge1$ and suppose that
\[
\Pi\Bigl(\sum_{j=1}^{k}v_j\otimes v_j\Bigr)=0.
\]
By linearity of $\ell$ and $\Pi$, condition (iv) gives
\[
0=\sum_{j=1}^{k}\ell\bigl(\Pi(v_j\otimes v_j)\bigr)\ge c\sum_{j=1}^{k}\norm{v_j}^2.
\]
Since $c>0$ and each squared norm is nonnegative, every $v_j=0$. Thus \Hc{k} holds for every $k\ge1$.

\smallskip\noindent
(iv)$\Rightarrow$(v): Condition (iv) immediately implies \Hc{1}. Let
\[
C:=\overline{\operatorname{cone}(W^+)}
=\overline{\Pi(\PSD(V))}.
\]
We show that $\ell(w)>0$ for every nonzero $w\in C$. Summing (iv) over a spectral decomposition gives
\[
\ell(\Pi(M))\ge c\,\tr M
\qquad(M\in\PSD(V)),
\]
so $\ell\ge0$ on $C$. If $w\in C$ satisfies $\ell(w)=0$, choose $M_k\succeq0$ with $\Pi(M_k)\to w$. Then
\[
0\le\tr M_k
\le c^{-1}\ell(\Pi(M_k))\longrightarrow0.
\]
Positivity gives $\norm{M_k}_{\mathrm{op}}\le\tr M_k$, so $M_k\to0$ and $w=0$. Thus $\ell$ is strictly positive on $C\setminus\{0\}$. Since $\ell(-w)=-\ell(w)$, the cone cannot contain both $w$ and $-w$ for any nonzero $w$. Hence $C$ is pointed.

\smallskip\noindent
(v)$\Rightarrow$(iii): if $0\ne M=\sum_jv_j\otimes v_j\in\ker\Pi\cap\PSD(V)$ with all $v_j\ne0$, then $w=\Pi(v_1\otimes v_1)\ne0$ by \Hc{1}, while $-w=\Pi(\sum_{j\ge2}v_j\otimes v_j)$ also lies in the cone, contradicting pointedness.
\end{proof}

\begin{samepage}
\begin{remark}[Interpretation]\label{rem:interpretation}
The equivalent conditions in Prop.~\ref{prop:equiv} ensure that images of nonzero positive semidefinite outer products cannot cancel to zero. Condition (iv) gives a scalar interpretation of this property: there is a single linear functional $\ell$ whose value on the polarized output controls the size of every positive semidefinite input. Indeed,
\[
\ell(\Pi(M))\ge c\,\tr M\qquad\text{for every }M\in\PSD(V).
\]
The same functional works for all such $M$, with a uniform constant $c>0$. For $\Pi=\id$, this is simply the ordinary trace, with $c=1$. For the second-moment operator arising from a de Finetti representation, $\tr M$ is the average squared constraint violation. The inequality therefore explains why the polarization constraint cannot conceal a nonzero average squared violation.

Condition (v) gives a complementary geometric interpretation. A cone is pointed if it contains no nonzero vector together with its negative, as illustrated by the nonnegative quadrant in $\R^2$. Here, \Hc{1} ensures that no nonzero outer product is mapped to zero, while pointedness excludes cancellation among nonzero elements of the closed convex cone generated by $W^+$. This includes elements obtained by positive sums and limits. The functional $\ell$ in (iv) makes this geometry explicit: it is strictly positive on every nonzero element of this cone, so two opposite nonzero vectors cannot both belong to it.
\end{remark}
\end{samepage}

The next proposition shows that, without further assumptions on $\Pi$, the corrected condition \Hc{D} cannot be replaced by \Hc{k} for any $k<D$.

\begin{proposition}\label{prop:sharp}
Let $V=\R^D$ with $D\ge2$. The map $\Pi_D$ defined in~\eqref{eq:PiD} satisfies \Hc{D-1} but not \Hc{D}.
\end{proposition}

\begin{proof}
Let
\[
M=\sum_{j=1}^{D-1}v_jv_j^{\mathsf T}
\]
and suppose that $\Pi_D(M)=0$. Since $M$ is symmetric, the definition of $\Pi_D$ gives $M=cI_D$, where $c=\tr M/D\ge0$. The range of $M$ is contained in $\operatorname{span}\{v_1,\ldots,v_{D-1}\}$, so $\rank M\le D-1$. Since $cI_D$ has rank $D$ whenever $c>0$, we obtain $c=0$. Taking the trace yields
\[
0=\tr M=\sum_{j=1}^{D-1}\norm{v_j}^2,
\]
so every $v_j=0$. Thus $\Pi_D$ satisfies \Hc{D-1}. On the other hand, the standard basis vectors satisfy
\[
\Pi_D\Bigl(\sum_{j=1}^{D}e_je_j^{\mathsf T}\Bigr)=\Pi_D(I_D)=0.
\]
Since these vectors are nonzero, \Hc{D} fails.
\end{proof}

Thus $D$ is the smallest $k$ such that \Hc{k} implies the conditions of Prop.~\ref{prop:equiv} for every $\Pi$. Carath\'eodory's theorem for cones in $\Sym^2(V)$ gives only $k=D(D+1)/2$; the value $D$ is the Carath\'eodory number of $\PSD(V)$~\cite{IL}. If $V$ is the space of real $d\times d$ matrices, as in \cite[p.~21]{PLG}, or of Hermitian $d\times d$ matrices, then $D=d^2$. The vectors $v_j$ in \Hc{k} range over $V$ and need not lie in the image of $f$.

\Needspace{11\baselineskip}
\begin{samepage}
The next proposition shows that the original two-vector condition suffices when the restriction of $\Pi$ to symmetric tensors has rank at most four, and gives a more general bound for higher ranks.

\begin{proposition}\label{prop:rank}
Let $m=\rank(\Pi|_{\Sym^2(V)})$ and
\begin{equation}\label{eq:kappa}
\kappa=\min\Bigl\{D,\Bigl\lfloor\frac{\sqrt{8m+9}-1}{2}\Bigr\rfloor\Bigr\}.
\end{equation}
Then \Hc{\kappa} implies the conditions of Prop.~\ref{prop:equiv}. In particular, \Hc{2} suffices if $m\le4$.
\end{proposition}
\end{samepage}

\begin{proof}
We prove the contrapositive: if (iii) of Prop.~\ref{prop:equiv} fails, then \Hc{\kappa} fails. Normalizing a nonzero positive semidefinite operator in $\ker\Pi$ by its trace shows that
\[
C=\{M\in\PSD(V):\Pi(M)=0,\ \tr M=1\}
\]
is nonempty. This set is compact and convex, so it has an extreme point $M$. We show that $\rank M\le\kappa$.

Let $U=\ran M$ and $r=\dim U$. The self-adjoint operators $H$ with $\ran H\subseteq U$ correspond to real symmetric $r\times r$ matrices and therefore form a space of dimension $r(r+1)/2$. On this space, $\Pi(H)=0$ imposes at most $m$ independent linear constraints, and $\tr H=0$ adds at most one. Consequently, if $r(r+1)/2>m+1$, the rank--nullity theorem gives a nonzero such $H$ with $\Pi(H)=0$ and $\tr H=0$.

Both $M$ and $H$ vanish on $U^\perp$, while $M|_U$ is positive definite. Thus, for sufficiently small $\varepsilon>0$, the restrictions $(M\pm\varepsilon H)|_U$ remain positive definite, and $M\pm\varepsilon H$ remain positive semidefinite on $V$. Moreover, $\Pi(M\pm\varepsilon H)=0$ and $\tr(M\pm\varepsilon H)=1$, so both perturbations belong to $C$. Since $H\ne0$, the identity
\[
M=\frac12(M+\varepsilon H)+\frac12(M-\varepsilon H)
\]
expresses $M$ as the midpoint of two distinct elements of $C$, contradicting extremality. Hence
\[
\frac{r(r+1)}2\le m+1,\qquad\text{that is,}\qquad r\le\frac{\sqrt{8m+9}-1}{2}.
\]
This is the rank bound of~\cite{Pataki}. Since $r$ is an integer and $r\le D$, we obtain $r\le\kappa$. A spectral decomposition $M=\sum_{j=1}^rv_j\otimes v_j$ with $v_j\ne0$ violates \Hc{r}. Padding this decomposition with $\kappa-r$ zero vectors gives a violation of \Hc{\kappa}, proving the contrapositive.

Finally, if $m\le4$, the bound $r(r+1)/2\le5$ forces $r\le2$, so any failure of (iii) already violates \Hc{2}.
\end{proof}

For $\Pi_3$ from the family~\eqref{eq:PiD}, $D=3$ and $m=5$, so $\kappa=3$ and Prop.~\ref{prop:rank} is consistent with Ex.~\ref{ex:feasibility}.

Whenever $\Pi$ fails \Hc{D}, the following construction produces an infeasible problem with a bicompatible sequence satisfying the polarization constraint.

\begin{proposition}\label{prop:necessary}
If $\ker\Pi\cap\PSD(V)\ne\{0\}$, there are state spaces $K_A,K_B$ and a map $f$, affine in each argument, with $\Fcal=\varnothing$ and a bicompatible sequence $(y_n)$ with $\Pi((f\otimes f)\Sreg_2(y_2))=0$.
\end{proposition}

\begin{proof}
Let $0\ne M\in\ker\Pi\cap\PSD(V)$ and write $M=\sum_{i=1}^r\lambda_iu_i\otimes u_i$ with $\lambda_i>0$ and orthonormal $u_i$. Let $K_A$ be a simplex with vertices $\sigma_1,\dots,\sigma_r$, let $K_B=\{b_0\}$ be a singleton, and let $f:K_A\times K_B\to V$ be the affine map determined by $f(\sigma_i,b_0)=\sqrt{r\lambda_i}\,u_i$. The factor $\sqrt r$ compensates for the uniform weights $1/r$ in the sequence below, so that its second-moment operator equals $M$.

Every $a\in K_A$ has barycentric coordinates $t_i\ge0$ with $\sum_i t_i=1$, so
\[
f(a,b_0)=\sum_{i=1}^r t_i f(\sigma_i,b_0)=\sum_{i=1}^r t_i\sqrt{r\lambda_i}\,u_i.
\]
If $f(a,b_0)=0$, orthonormality forces every $t_i=0$, contradicting $\sum_i t_i=1$. Thus $f$ has no zero on $K_A\times K_B$, so $\Fcal=\varnothing$.

Now define
\[
y_n=\frac1r\sum_{i=1}^r\sigma_i^{\otimes n}\otimes b_0^{\otimes n}.
\]
Each summand is invariant under independent permutations of the $A$- and $B$-copies, and marginalization preserves the uniform mixture, so $(y_n)$ is bicompatible. Moreover,
\[
\begin{aligned}
(f\otimes f)\Sreg_2(y_2)
&=\frac1r\sum_{i=1}^r f(\sigma_i,b_0)\otimes f(\sigma_i,b_0)\\
&=\sum_{i=1}^r\lambda_iu_i\otimes u_i=M.
\end{aligned}
\]
Since $M\in\ker\Pi$, the polarization constraint is satisfied:
\[
\Pi\bigl((f\otimes f)\Sreg_2(y_2)\bigr)=\Pi(M)=0.\qedhere
\]
\end{proof}

Hence condition \Hc{D} is necessary for \cite[Thm.~3]{PLG} to hold for all state spaces and all maps $f$. It need not be necessary for a fixed problem.

\section{The corrected theorems}\label{sec:theorems}
We now show that replacing \Hc{2} by \Hc{D} restores the characterization of feasibility in terms of bicompatible sequences satisfying the polarization constraint.

\begin{theorem}[Correction of {\cite[Thm.~3]{PLG}}]\label{thm:polar}
Let $\Pi$ satisfy \Hc{D}. Then $\Fcal\ne\varnothing$ if and only if there is a bicompatible sequence $(y_n)$ with
\begin{equation}\label{eq:constraint}
\Pi\bigl((f\otimes f)\Sreg_2(y_2)\bigr)=0.
\end{equation}
Every representing measure $\mu$ of such a sequence satisfies $\supp\mu\subseteq\Fcal$.
\end{theorem}

\begin{proof}
If $(a,b)\in\Fcal$, the sequence $y_n=a^{\otimes n}\otimes b^{\otimes n}$ is bicompatible and satisfies
\[
(f\otimes f)\Sreg_2(y_2)=f(a,b)\otimes f(a,b)=0.
\]
Hence the polarization constraint holds.

Conversely, let $(y_n)$ be a bicompatible sequence satisfying~\eqref{eq:constraint}, and let $\mu$ be any representing measure, whose existence follows from the de Finetti theorem. Define
\[
M:=(f\otimes f)\Sreg_2(y_2)
=\int f(a,b)\otimes f(a,b)\,d\mu(a,b).
\]
For every $z\in V$,
\[
\langle z,Mz\rangle=\int\langle z,f(a,b)\rangle^2\,d\mu(a,b)\ge0,
\]
so $M$ is positive semidefinite. The polarization constraint gives $\Pi(M)=0$, and \Hc{D}, through Prop.~\ref{prop:equiv}(iii), therefore implies $M=0$. Taking the trace yields
\[
0=\tr M=\int\norm{f(a,b)}^2\,d\mu(a,b).
\]
Since the integrand is nonnegative, $f=0$ $\mu$-almost everywhere, so $\mu(\Fcal)=1$. Continuity of $f$ makes $\Fcal$ closed, hence $\supp\mu\subseteq\Fcal$. In particular, $\Fcal\ne\varnothing$. As $\mu$ was arbitrary, the support inclusion holds for every representing measure.
\end{proof}

\begin{remark}\label{rem:maps}
The polarization maps proposed in \cite[p.~21]{PLG} already satisfy the corrected hypothesis \Hc{D}. This can be verified using the scalar criterion in Prop.~\ref{prop:equiv}(iv), with $c=1$: for $\Pi=\id_{V\otimes V}$ take $\ell(v\otimes w)=\langle v,w\rangle$; for $V=\R^{d\times d}$ with the Frobenius inner product and the maps
\[
\begin{aligned}
\Pi_{\rm HS}(A\otimes B)&=\tr(A^{\mathsf T}B)\in\R,\\
\Pi_{\rm M}(A\otimes B)&=A^{\mathsf T}B\in\R^{d\times d},
\end{aligned}
\]
take $\ell=\id_\R$ and $\ell=\tr$, respectively. In each case $\ell(\Pi(v\otimes v))=\norm v^2$. The codomain of $\Pi_{\rm M}$ is $\R^{d\times d}$, since $A^{\mathsf T}B$ need not be symmetric.

These functionals also give a direct replacement for the dual-cone argument in~\cite{PLG}. Indeed, with $M$ and $\mu$ as in the proof of Thm.~\ref{thm:polar},
\[
0=\ell(\Pi(M))=\int\ell\bigl(\Pi(f(a,b)\otimes f(a,b))\bigr)\,d\mu(a,b)\ge c\int\norm{f(a,b)}^2\,d\mu(a,b),
\]
so $f=0$ $\mu$-almost everywhere.
\end{remark}

Thm.~\ref{thm:polar} assumes a full infinite bicompatible sequence. To prove convergence of the hierarchy, we must first construct such a sequence from minimizers at finite levels. This requires a compactness argument in addition to the corrected polarization condition \Hc{D}.

The proof of \cite[Thm.~5]{PLG}, pp.~24--25, takes a limit of the two-copy marginals only.\footnote{In the final line of this proof, we interpret $p(x_2^\star)$ as $p(x_1^\star)$, where $x_1^\star$ denotes the corresponding one-copy marginal, consistently with \cite[Eq.~(28)]{PLG}.} The extension property used there guarantees symmetry under permutations of the pairs $A_iB_i$. Bicompatibility, however, requires independent permutations of the $A$- and $B$-copies at every order, together with compatibility under marginalization. Ex.~\ref{ex:parity} shows that pair-symmetric extensions and independent symmetry at order two do not suffice.

Thm.~\ref{thm:convergence} closes this gap by extracting limits of all fixed-order marginals along a common subsequence. These limits form a bicompatible sequence, to which de Finetti applies. Although the finite-level operators need not be positive semidefinite, the resulting second-moment operator is. Condition \Hc{D} then forces it to vanish, so the representing measure is supported on feasible pairs. This proves asymptotic convergence of the optimal values without requiring any finite relaxation to be exact.

\begin{example}\label{ex:parity}
Let $K_A=K_B$ be the one-dimensional probability simplex on $\{0,1\}$, the state space of a classical bit. Then $K_A\tmax K_B$ is the three-dimensional probability simplex on $\{0,1\}^2$. Let $\rho_+=\frac12(\delta_{00}+\delta_{11})$ and $\rho_-=\frac12(\delta_{01}+\delta_{10})$, where $\delta_{ab}$ is the deterministic distribution assigning probability $1$ to $(a,b)$. Under $\rho_+$, the two bits are uniformly distributed and always agree; under $\rho_-$, they are uniformly distributed and always disagree. Define the two-copy state
\begin{equation}\label{eq:parity}
x_2=\tfrac12\rho_+^{\otimes2}+\tfrac12\rho_-^{\otimes2}.
\end{equation}
A bit string has even or odd parity according to whether its number of ones is even or odd. The mixture chooses with equal probability whether both pairs agree or both pairs disagree. In either case, the total number of ones is even, and each of the eight strings with this property has probability $1/8$. Thus $\Sreg_2^{-1}x_2$ is uniform on the even-parity strings $(a_1,a_2,b_1,b_2)\in\{0,1\}^4$.\footnote{Here $\Sreg_2^{-1}$ reorders the tensor factors from $A_1B_1A_2B_2$ to $A_1A_2B_1B_2$.} Swapping the $A$-copies or the $B$-copies preserves parity, so this distribution is invariant under these independent permutations. Moreover, $x_2\in\Ext_\infty(K_A\tmax K_B,2)$, being the two-copy marginal of $\frac12\rho_+^{\otimes n}+\frac12\rho_-^{\otimes n}$ for every $n$.

We now show that, despite these properties, $x_2$ cannot be the two-copy marginal of a full bicompatible sequence. Suppose that $\Sreg_2(y_2)=x_2$ for a bicompatible sequence $(y_n)$ with representing measure $\mu$, and write $a(0)$ for the probability of the outcome $0$ under $a\in K_A$. The $A_1A_2$-marginal of $y_2$ is uniform because each choice of the two $A$-bits has exactly two even-parity completions by the $B$-bits. Hence $\int a(0)\,d\mu=\frac12$ and $\int a(0)^2\,d\mu=\frac14$. Consequently,
\[
\int\bigl(a(0)-\tfrac12\bigr)^2\,d\mu
=\int a(0)^2\,d\mu-\int a(0)\,d\mu+\tfrac14
=\tfrac14-\tfrac12+\tfrac14=0.
\]
This vanishing variance forces the local probability $a(0)$ to equal $1/2$ almost everywhere; similarly, $b(0)=1/2$ almost everywhere. Hence the representing measure is concentrated on the single pair of uniform local distributions. The de Finetti representation would therefore make all four bits independent and uniform, assigning probability $1/16$ to every string. This contradicts the absence of odd-parity strings in $x_2$. Sec.~\ref{sec:lp} gives a finite certificate.
\end{example}

Ex.~\ref{ex:parity} shows that the extension properties used in~\cite{PLG} do not suffice to establish bicompatibility. The proof of Thm.~\ref{thm:convergence} supplies this missing step by extracting limits of all fixed-order marginals along a common subsequence and verifying that these limits form a bicompatible sequence.
The following theorem shows that, under \Hc{D}, the optimal values of the hierarchy with the polarization constraint converge monotonically to the true optimum whenever the original problem is feasible.

\begin{theorem}[Correction of the first assertion of {\cite[Thm.~5]{PLG}}]\label{thm:convergence}
Let $\Fcal\ne\varnothing$ and let $\Pi$ satisfy \Hc{D}. Then the infima $p_n^\Pi$ are attained, and
\[
p_n^\Pi\le p_{n+1}^\Pi\le p^*\quad(n\ge2),\qquad\lim_{n\to\infty}p_n^\Pi=p^*.
\]
\end{theorem}

\begin{proof}
The set $\Kn n$ is compact, the constraints in~\eqref{eq:hierarchy} are closed, and the objective is linear. By~\eqref{eq:upper}, $\Omega_n^\Pi\ne\varnothing$ and $p_n^\Pi\le p^*$; hence the infimum is attained.

The partial trace $R_{n+1,n}$ preserves invariance under independent permutations and satisfies $R_{n,k}R_{n+1,n}=R_{n+1,k}$; hence it maps $\Omega_{n+1}^\Pi$ into $\Omega_n^\Pi$ and preserves the objective, so $p_n^\Pi\le p_{n+1}^\Pi$. Therefore $p_n^\Pi$ converges to some $L\le p^*$. It remains to prove $L\ge p^*$, which we do by constructing a limiting representing measure supported on $\Fcal$.

Choose minimizers $y^{(n)}\in\Omega_n^\Pi$ and define the finite-level operators
\[
M_n=(f\otimes f)\Sreg_2(R_{n,2}y^{(n)}),\qquad \Pi(M_n)=0.
\]
For each $k$, the marginals $R_{n,k}y^{(n)}$, $n\ge k$, lie in the compact set $\Kn k$. First choose a subsequence along which the one-copy marginals converge, then a further subsequence along which the two-copy marginals converge, and continue. Taking the $j$-th term from the $j$-th subsequence gives a single sequence of levels $n_j\to\infty$ along which every fixed-order marginal converges:
\begin{equation}\label{eq:diagonal}
R_{n_j,k}\,y^{(n_j)}\longrightarrow z_k\in\Kn k\qquad(k\ge1).
\end{equation}
Using the same sequence of levels for every marginal order allows compatibility to pass to the limit. Each $z_k$ is invariant under independent permutations, as a limit of invariant elements, and for $k'<k$,
\[
R_{k,k'}z_k=\lim_{j\to\infty}R_{k,k'}R_{n_j,k}\,y^{(n_j)}=\lim_{j\to\infty}R_{n_j,k'}\,y^{(n_j)}=z_{k'}.
\]
Thus $(z_k)$ is bicompatible. Let $\mu$ be a representing measure. By continuity and the de Finetti representation,
\[
\begin{aligned}
M_{n_j}&\longrightarrow M_\infty:=(f\otimes f)\Sreg_2(z_2),\\
M_\infty&=\int f(a,b)\otimes f(a,b)\,d\mu(a,b)\in\PSD(V).
\end{aligned}
\]
Consequently, $\Pi(M_\infty)=\lim_j\Pi(M_{n_j})=0$, and Prop.~\ref{prop:equiv} gives $M_\infty=0$. As in the proof of Thm.~\ref{thm:polar}, it follows that $f=0$ $\mu$-almost everywhere and $\supp\mu\subseteq\Fcal$. Since $\mu$ is supported on $\Fcal$, we have $p(a,b)\ge p^*$ $\mu$-almost everywhere. By continuity of the objective,
\[
L=\lim_jp_{n_j}^\Pi=p(z_1)=\int p(a,b)\,d\mu(a,b)\ge p^*.
\]
Together with $L\le p^*$, this proves $L=p^*$.
\end{proof}

\begin{remark}\label{rem:convergence}
(a) Thm.~\ref{thm:convergence} asserts convergence of the optimal values; it does not assert convergence of minimizers or equality $p_n^\Pi=p^*$ at a finite level. For $\Pi=\id$ in Ex.~\ref{ex:optimization}, \eqref{eq:idvalue} gives
\[
p_n^{\id}=-\frac1n<0=p^*\qquad(n\ge2),
\]
so the values converge to $p^*$ without attaining it at any finite level.

(b) Let $\Pi$ satisfy \Hc{D} and $\Fcal=\varnothing$. Then $\Omega_n^\Pi=\varnothing$ for some $n$, hence for all larger $n$ since $R_{n+1,n}(\Omega_{n+1}^\Pi)\subseteq\Omega_n^\Pi$: otherwise the diagonal argument, applied to arbitrary $y^{(n)}\in\Omega_n^\Pi$, yields a bicompatible sequence satisfying~\eqref{eq:constraint}, contradicting Thm.~\ref{thm:polar}. The argument gives no bound on $n$.
\end{remark}

The stronger hierarchy of \cite[Eq.~(29)]{PLG} does not use an auxiliary polarization map and requires no modification. Its convergence follows from Thm.~\ref{thm:convergence}, as the following corollary shows.

\begin{corollary}[Second assertion of {\cite[Thm.~5]{PLG}}]\label{cor:strong}
Let $\Fcal\ne\varnothing$. Then the infima $q_n$ are attained, and
\[
p_n^{\id}\le q_n\le q_{n+1}\le p^*\quad(n\ge2),\qquad\lim_{n\to\infty}q_n=p^*.
\]
If $\Fcal=\varnothing$, then $\Omega_n=\varnothing$ for all sufficiently large $n$.
\end{corollary}

\begin{proof}
The stronger constraint applies $f$ to the first $AB$-pair while retaining all remaining pairs. Applying $f$ to the second pair and tracing out the others gives the second-moment constraint with $\Pi=\id$. Explicitly, applying
\[
\id_V\otimes f\otimes(\one_{K_A}\otimes\one_{K_B})^{\otimes(n-2)}
\]
to~\eqref{eq:strong} gives $(f\otimes f)\Sreg_2(R_{n,2}y)=0$, as in \cite[Eq.~(22)]{PLG}; hence $\Omega_n\subseteq\Omega_n^{\id}$ and $p_n^{\id}\le q_n$.

Applying
\[
\id_V\otimes\id_E^{\otimes(n-1)}\otimes(\one_{K_A}\otimes\one_{K_B})
\]
to~\eqref{eq:strong} at level $n+1$ shows $R_{n+1,n}(\Omega_{n+1})\subseteq\Omega_n$. Since marginalization also preserves the objective, it follows that $q_n\le q_{n+1}$. Attainment and $q_n\le p^*$ follow as in the proof of Thm.~\ref{thm:convergence}.

Since $\id$ satisfies \Hc{D} (Rem.~\ref{rem:maps}), Thm.~\ref{thm:convergence} gives $p_n^{\id}\to p^*$. The bounds $p_n^{\id}\le q_n\le p^*$ therefore imply $q_n\to p^*$. Rem.~\ref{rem:convergence}(b), together with $\Omega_n\subseteq\Omega_n^{\id}$, gives the last assertion.
\end{proof}

\begin{remark}[Application to nonnegative matrix rank]\label{rem:application}
The insufficiency of \Hc{2} does not invalidate the finite-level infeasibility conclusions in \cite[Fig.~1]{PLG}. The blue and red regions record infeasibility at $n=3$ for the stronger hierarchy~\eqref{eq:strong} and the hierarchy with the polarization constraint~\eqref{eq:hierarchy}, respectively.\footnote{The polarization map used for the red region is not explicitly specified in~\cite{PLG}. All three maps proposed there satisfy \Hc{D}; see Rem.~\ref{rem:maps}. Thus, if one of these maps was used, our corrected convergence result applies to that hierarchy as well.} If the required nonnegative factorization exists, every relaxation is feasible; see~\eqref{eq:upper}. Hence infeasibility at any level rules out such a factorization. This conclusion does not require asymptotic convergence.

These certificates alone do not establish that every infeasible instance will eventually be detected. For the stronger hierarchy, this guarantee follows from Cor.~\ref{cor:strong}; for the hierarchy with the polarization constraint, it follows under \Hc{D} from Rem.~\ref{rem:convergence}(b).
\end{remark}

\begin{remark}[Several state spaces and higher-degree polynomials]\label{rem:general}
The counterexamples and corrections above also apply to the extensions in Sec.~2.1 of~\cite{PLG}, which allow several state spaces $K_1,\ldots,K_m$ and higher-degree polynomial objectives and constraints~\cite[Eqs.~(25), (26)]{PLG}. Convergence is stated to follow by the same reasoning~\cite[p.~23]{PLG}. For two state spaces and $\vec\ell=(1,1)$ in \cite[Eq.~(23)]{PLG}, \cite[Eq.~(25)]{PLG} is \cite[Eq.~(28)]{PLG}, so Ex.~\ref{ex:optimization} applies.

Choose $\ell\ge1$ such that $f$ and $p$ have degree at most $\ell$ in each local state. These polynomials can be represented by linear maps acting on $\ell$ copies of each state space, using unit functionals to accommodate lower-degree terms. Consequently, for $n\ge2\ell$, the objective uses $\ell$-copy marginals, while the squared-polarization constraint uses $2\ell$-copy marginals. The stronger hierarchy retains its full $n$-copy constraint.

The compactness argument again extracts limits of all fixed-order marginals along a common subsequence. Compatibility and invariance under independent permutations of the copies of each state space pass to these limits, so the multipartite de Finetti representation gives a measure $\mu$ on $K_1\times\cdots\times K_m$. On each product component of this representation, the multilinearized constraint map evaluates to the original polynomial $f$. Thus, writing $\mathbf a=(a_1,\ldots,a_m)$, the limiting polarization constraint is again
\[
\Pi\left(\int f(\mathbf a)\otimes f(\mathbf a)\,d\mu(\mathbf a)\right)=0,
\]
and the limiting objective is $\int p(\mathbf a)\,d\mu(\mathbf a)$. The operator inside $\Pi$ is positive semidefinite. Condition \Hc{D} forces it to vanish, so the representing measure is supported on solutions of the original polynomial constraint. Hence the feasibility and convergence arguments of this section carry over.
\end{remark}

\section{Linear-programming verification}\label{sec:lp}
We first derive exact formulas for the hierarchy values in Ex.~\ref{ex:optimization} and compare them with numerical solutions of the corresponding linear programs. We then give a finite certificate for the failure of bicompatible extendibility in Ex.~\ref{ex:parity} and quantify the distance to the set of extendible marginals.

\subsection{Exact hierarchy values}
In Ex.~\ref{ex:optimization}, $K_A$ is a simplex with vertices $\sigma_0,\dots,\sigma_3$, so $\Kn n$ is the set of probability distributions on $\{0,1,2,3\}^n$, where the label $i$ stands for $\sigma_i$. A permutation orbit is specified by its occupation vector $k=(k_0,k_1,k_2,k_3)$, where $k_i$ counts the positions carrying label $i$ and $\sum_i k_i=n$. For example, the strings $(0,1,1)$, $(1,0,1)$, and $(1,1,0)$ have the same occupation vector $k=(1,2,0,0)$. Permutation invariance assigns them equal probabilities, so we need only record their total probability $\alpha_k$. In general, a distribution invariant under $\Sn_n$ is determined by these orbit masses $\alpha_k\ge0$, $\sum_k\alpha_k=1$; there are $\binom{n+3}3$ orbits.

Let $s=k_1+k_2+k_3$ count the positions carrying a nonzero label, and let $\mathbb E_\alpha$ denote expectation with respect to $\alpha$. The objective and the matrix $M_n=(f\otimes f)\Sreg_2(R_{n,2}y)$ are
\begin{equation}\label{eq:count}
p(R_{n,1}y)=-\frac{\mathbb E_\alpha[s]}{n},\qquad
(M_n)_{ij}=\frac{\mathbb E_\alpha[k_i(k_j-\delta_{ij})]}{n(n-1)}\qquad(1\le i,j\le3).
\end{equation}
For a fixed occupation vector $k$, a single position has a nonzero label with probability $s/n$. The two-copy marginal samples two distinct positions without replacement. There are $n(n-1)$ ordered choices; the number with labels $i,j$ is $k_i k_j$ if $i\ne j$, and $k_i(k_i-1)$ if $i=j$. Averaging these expressions over $\alpha$ gives~\eqref{eq:count}. Thus $p_n^\Pi$ is the value of the linear program that minimizes $-\mathbb E_\alpha[s]/n$ subject to $\Pi(M_n)=0$.

At a finite level, $M_n$ records correlations between two distinct copies and need not be positive semidefinite. The positivity of the second moment in~\eqref{eq:moment} follows from the de Finetti representation of a full bicompatible sequence.

For $\Pi=\Pi_3$, the matrix $M_n$ is symmetric, so $\Pi_3(M_n)=0$ means $M_n=\frac13(\tr M_n)I_3$. The off-diagonal constraints $\mathbb E_\alpha[k_ik_j]=0$, $i\ne j$, allow at most one distinct nonzero label per orbit in the support of $\alpha$. This label may nevertheless occur in all $n$ positions. The value $-1$ in~\eqref{eq:badvalue} requires $s=n$ on the support, and the diagonal constraints then force mass $\frac13$ on each of the orbits $k=(0,n,0,0),(0,0,n,0),(0,0,0,n)$. Hence the minimizer is unique; the gap in~\eqref{eq:badvalue} is not due to multiple minimizers.

For $\Pi=\id$, $M_n=0$ gives
\[
\sum_{i,j}(M_n)_{ij}=\frac{\mathbb E_\alpha[s(s-1)]}{n(n-1)}=0,
\]
so $s\le1$ on the support: at most one position may carry a nonzero label. Its contribution to the one-copy marginal is therefore at most $1/n$, giving $p(R_{n,1}y)\ge-1/n$, with equality for the orbit $k=(n-1,1,0,0)$. Hence
\begin{equation}\label{eq:idvalue}
p_n^{\id}=-\frac1n\longrightarrow0=p^*.
\end{equation}
The same lower bound has a dual certificate. After averaging, the following identity expresses the objective minus $-1/n$ as a nonnegative quantity plus terms that vanish under $M_n=0$:
\[
-\frac sn+\frac1n+(n-1)\sum_{i,j=1}^3\frac{k_i(k_j-\delta_{ij})}{n(n-1)}=\frac{(s-1)^2}{n}\ge0;
\]
averaging over $\alpha$ and imposing $M_n=0$ gives $p(R_{n,1}y)\ge-1/n$. The minimizers are the mixtures of the three orbits with $s=1$.

\begin{table}[htbp]
\centering
\setlength{\belowcaptionskip}{8pt}
\caption{Values of the linear programs of Sec.~\ref{sec:lp} for Ex.~\ref{ex:optimization}, rounded to five decimal places; $p^*=0$. The exact values for all $n\ge2$ are given by~\eqref{eq:badvalue} and~\eqref{eq:idvalue}.}
\label{tab:numerics}
\begin{tabular}{rrrr}
\toprule
$n$ & Orbits & $p_n^{\Pi_3}$ & $p_n^{\id}$\\\midrule
2  & 10   & $-1.00000$ & $-0.50000$\\
4  & 35   & $-1.00000$ & $-0.25000$\\
8  & 165  & $-1.00000$ & $-0.12500$\\
16 & 969  & $-1.00000$ & $-0.06250$\\
32 & 6545 & $-1.00000$ & $-0.03125$\\\bottomrule
\end{tabular}
\end{table}

\subsection{A finite extension certificate}
We now show that the obstruction in Ex.~\ref{ex:parity} already appears when extending to three copies of each party. Suppose that $\Sreg_2^{-1}x_2$ is the marginal on $A_1A_2B_1B_2$ of a distribution on three A bits and three B bits that is invariant under independent permutations. Let $Q_{kl}$ be the probability that $k$ of the three A bits and $l$ of the three B bits equal $1$. If three bits contain $k$ ones, the probability of finding $r$ ones when choosing two positions uniformly is
\[
\eta_r(k)=\frac{\binom kr\binom{3-k}{2-r}}{3}.
\]
The denominator is the number $\binom32=3$ of choices of two positions; binomial coefficients outside their usual range are taken to be zero. By invariance under independent permutations, the probability that $r$ of the first two A bits and $t$ of the first two B bits equal $1$ is
\[
T_{rt}=\sum_{k,l=0}^3Q_{kl}\,\eta_r(k)\,\eta_t(l).
\]
Fix $(k,l)$ and put $h=\eta_1(k)$, $g=\eta_1(l)$. Since $\eta_0+\eta_2=1-\eta_1$, the contribution of $(k,l)$ to $T_{01}+T_{10}+T_{12}+T_{21}-T_{11}$ is $Q_{kl}$ times
\[
g(1-h)+h(1-g)-hg=h+g-3hg,
\]
which is nonnegative because $h,g\in\{0,\frac23\}$. Hence
\begin{equation}\label{eq:separator}
T_{11}\le T_{01}+T_{10}+T_{12}+T_{21}.
\end{equation}
The right-hand side is the probability of odd total parity among the four bits. Thus any such extension requires this probability to be at least the probability that each local pair contains exactly one $1$. For $x_2$, $T_{11}=\frac12$, $T_{00}=T_{02}=T_{20}=T_{22}=\frac18$ and all other $T_{rt}$ vanish, which violates~\eqref{eq:separator}. Thus the corresponding feasibility linear program in the $16$ variables $Q_{kl}$ is infeasible.

The obstruction persists under small perturbations: every two-copy marginal admitting such a three-copy extension is at total variation distance at least $\frac14$ from $x_2$. We first derive this lower bound from~\eqref{eq:separator}, then give an extension attaining it. Both $x_2$ and every candidate marginal are uniform within each orbit of independent permutations, so their total variation distance equals that between their distributions of occupation counts. Denote these distributions by $T^0$ and $T$, respectively, and use the convention
\[
d_{\mathrm{TV}}(T^0,T)=\frac12\sum_{r,t=0}^2|T^0_{rt}-T_{rt}|.
\]
\Needspace{7\baselineskip}
Let $\varphi$ take the value $1$ at $(1,1)$, the value $-1$ at $(0,1),(1,0),(1,2),(2,1)$, and zero elsewhere. Then $\mathbb E_{T^0}\varphi=\frac12$, whereas~\eqref{eq:separator} gives $\mathbb E_T\varphi\le0$. Since $|\varphi|\le1$, the distance is at least $\frac14$:
\[
\frac12\le\bigl|\mathbb E_{T^0}\varphi-\mathbb E_T\varphi\bigr|
\le 2d_{\mathrm{TV}}(T^0,T).
\]

To attain the lower bound, take
\[
Q_{kl}=\begin{cases}
\frac9{64}, & (k,l)\in\{1,2\}^2,\\[2pt]
\frac7{64}, & (k,l)\in\{0,3\}^2,\\[2pt]
0, & \text{otherwise}.
\end{cases}
\]
The resulting distribution $T^\star$ retains the four corner probabilities $\frac18$, has $T^\star_{11}=\frac14$, and assigns probability $\frac1{16}$ to each of the four remaining cells. Hence
\[
d_{\mathrm{TV}}(T^0,T^\star)=\frac12\left(\frac14+4\cdot\frac1{16}\right)=\frac14.
\]

\section*{Acknowledgements}
The author thanks Martin Pl\'avala and David Gross for discussions and for supporting this manuscript. The author acknowledges support from the European Research Council (ERC Grant Agreement No.~948139) and the Excellence Cluster -- Matter and Light for Quantum Computing (ML4Q-2).

\section*{Data availability}
The code, the archived output of the linear programs, and a script verifying the certificates of Sec.~\ref{sec:lp} in exact rational arithmetic are provided as supplementary material.

\section*{Declaration of generative AI and AI-assisted technologies in the manuscript preparation process}
Generative AI tools supported manuscript preparation, mathematical analysis, and code development. Any content not developed by the author was reviewed by the author, who takes full responsibility for the published work.

\end{document}